\documentclass{amsart}

\usepackage{amsmath}%
\usepackage{amsfonts}%
\usepackage{amssymb}%
\usepackage{graphicx}
\usepackage{graphicx, xcolor}              % to include figures
\usepackage{amsmath}               % great math stuff
\usepackage{amsfonts}              % for blackboard bold, etc
\usepackage{amsmath}
\usepackage{amsthm}  
 \usepackage{float} 
\usepackage{caption}
\usepackage{wrapfig}
\usepackage{subcaption}
\usepackage{color}

\usepackage{mathrsfs}
\usepackage{amsbsy}
\usepackage{url}
\usepackage{cite}
\newtheorem{theorem}{Theorem}

\theoremstyle{plain}

\newtheorem{conjecture}{Conjecture}
\newtheorem{corollary}{Corollary}

\newtheorem{definition}{Definition}

\newtheorem{lemma}{Lemma}

\numberwithin{equation}{section}

\begin{document}
\title{Parity in Knotoids}

\author{Neslihan G{\"u}g{\"u}mc{\"u}}
\author{Louis H.Kauffman}

\address{Neslihan G{\"u}g{\"u}mc{\"u}}
\address{Louis H.Kauffman:Department of Mathematics, Statistics and Computer
Science, University of Illinois at Chicago, 851 South Morgan St., Chicago
IL 60607-7045, U.S.A. and
Department of Mechanics and Mathematics
Novosibirsk State University
Novosibirsk
Russia}
\email{nesli@central.ntua.gr} \email{kauffman@math.uic.edu}
\maketitle

\begin{abstract}
This paper investigates the parity concept in knotoids in $S^2$ and in $\mathbb{R}^2$ in relation with virtual knots. We show that the virtual closure map is not surjective and give specific examples of virtual knots that are not in the image. We introduce a planar version of the parity bracket polynomial for knotoids in $\mathbb{R}^2$. By using the Nikonov/Manturov theorem on minimal diagrams of virtual knots we prove a conjecture of Turaev showing that minimal diagrams of knot-type knotoids have zero height.  
\end{abstract}

\section{Introduction}
Knotoids [24] are knot diagrams with free ends (with the ends possibly in distinct planar regions of the diagram), taken up to classical Reidemeister moves. The moves are not allowed to pass an arc across an end of the diagram. Turaev suggests knotoids as abbreviated knot diagrams that provide an ease of computation of knot invariants. In fact, knotoids can be considered to be a significant extension of classical knot theory, and are worth studying for their own properties. Knotoids in $S^2$ also extend knots in $S^2 \times I$ and bring up many interesting features and problems such as the existence of parity in knotoid diagrams.

After an overview on knotoids and virtual knots and a discussion on the knotoid invariants, the {\it loop bracket polynomial} and the {\it arrow polynomial} in Section \ref{sec:basics}, we study the Gauss codes of knotoid diagrams in Section \ref{sec:parity}. We observe that during a traverse (a walk along the diagram) of a knotoid diagram in $S^2$ and in $\mathbb{R}^2$, each crossing can be called odd or even  according to whether one meets an even or an odd number of other crossings in the walk from the given crossing and returning to it. See Figure \ref{fig:loop}. Note that the parity of the crossing $c$ in Figure \ref{fig:loop} is an odd crossing.
\begin{figure}[H]
\centering
\includegraphics[width=.25\textwidth]{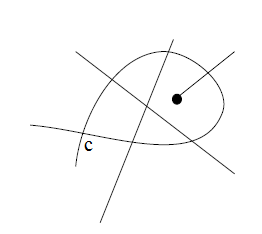}
\caption{A loop at an odd crossing enclosing an endpoint}
\label{fig:loop}
\end{figure}
Virtual knots represented by knots in thickened surfaces of arbitrary genus, also admit parity of this kind. Parity is an important theme in virtual knot theory and is related to a number of virtual knot invariants such as the odd writhe, the parity bracket polynomial, the affine index polynomial and the arrow polynomial. 

In Section \ref{sec:virtclosure} we study knotoids in relation with virtual knots. The endpoints of a knotoid can be connected to obtain a virtual knot called the virtual closure of the knotoid. Knotoids are related to virtual knots of genus at most $1$ through the virtual closure map. Many knotoid invariants are defined by applying virtual knot invariants to the virtual closure of the knotoid. In this way we invoke virtual parity in studying knotoids, and in many cases the induced invariants can be defined directly in terms of knotoids. This is the case for the odd writhe \cite{GK1}, the parity bracket polynomial \cite{GK1}, the arrow polynomial \cite{GK1} and the affine index polynomial \cite{GK1}. Then issues of parity are referred directly to the knotoid diagrams.
	
In Section \ref{sec:surj} we show that the virtual closure map is a non-surjective map by examining the surface bracket states of virtual knots lying in the image of the virtual closure map. The non-surjectivity of the virtual closure map for knotoids is important because the apparent structure of knotoids is quite different from the structure of virtual knots of genus no more than one. In particular we have conjectured that the Jones polynomial for knotoids detects the trivial knotoid \cite{GK1}. This is not true for genus one virtual knots. Our examples show that an infinite class of genus one virtual knots of unit Jones polynomial are not in the image of the virtual closure map on knotoids. The problem remains open to fully understand the image of the virtual closure map. 

Manturov shows that minimal genus representations of virtual knots admit minimal number of crossings by using a parity projection map. By using this result and the virtual closure we prove a conjecture of Turaev in Section \ref{sec:minimal}: Minimal diagrams of knot-type knotoids (knotoids admitting diagrams whose endpoints lie in the same planar region) have zero complexity (we often refer to the complexity of a knotoid as its height). This gives us the following result: The crossing number of a knot-type knotoid is equal to the crossing number of the knot that is the closure of the knotoid.

\section{Knotoids and virtual knots}
\subsection{Basics on knotoids}\label{sec:basics}
%The theory of knotoids, introduced by Turaev \cite{Tu} is a natural generalization of classical knot theory.
A \textit{knotoid diagram} $K$ in an oriented surface $\Sigma$ \cite{Tu} is an immersion of the unit interval $[0,1]$ into $\Sigma$ with a finite number of transversal double points, each of which is endowed with over or under information and called a 
\begin{wrapfigure}{r}{0.3\textwidth}
  \centering
	\hspace{-1cm}
	\includegraphics[width=0.3\textwidth]{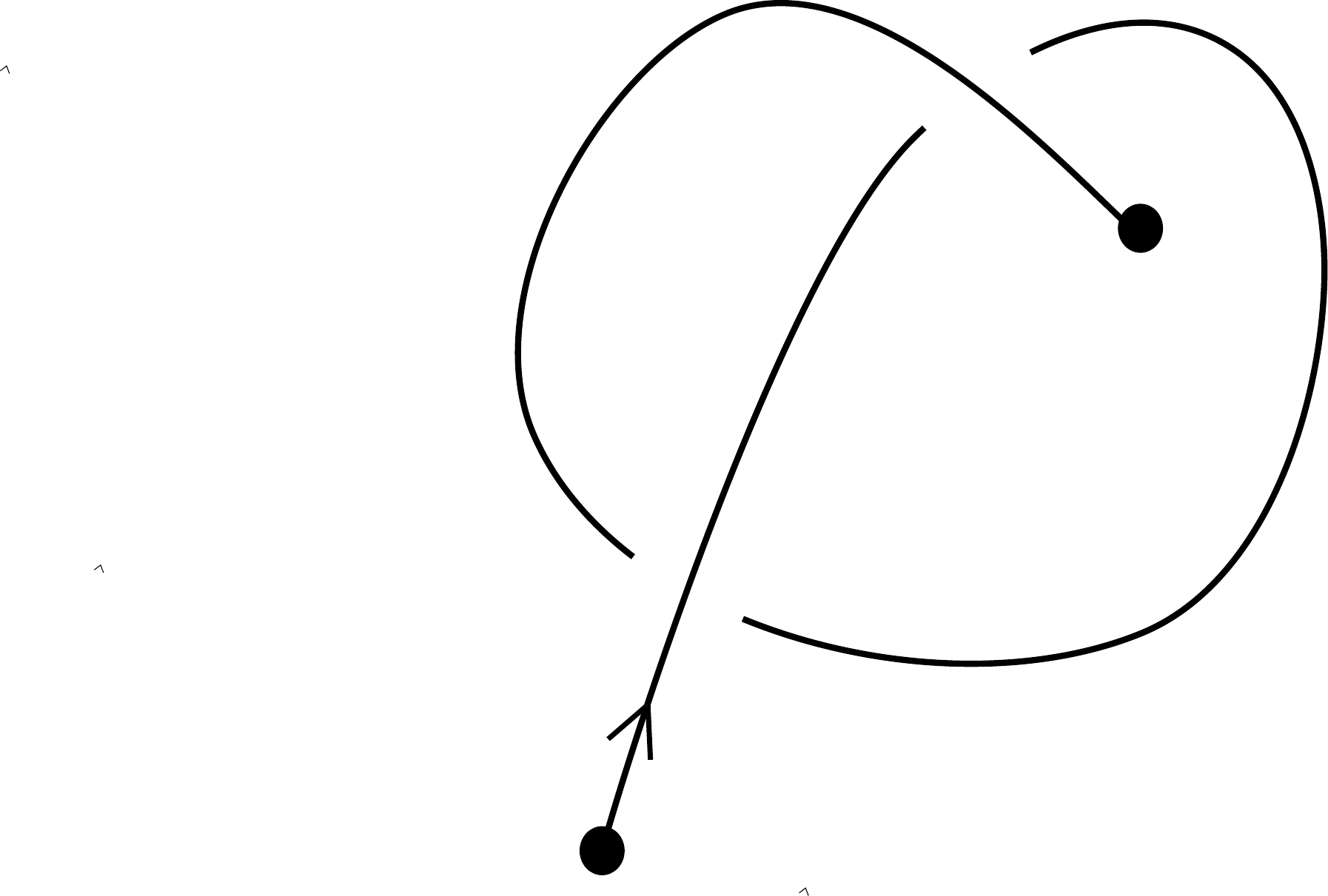}
%\end{center}
\vspace{-12.5pt}
  \hspace{-12.5pt}
\caption{\small A knotoid diagram}
   \vspace{-10pt}
	 \label{fig:kntd}
\end{wrapfigure}
 {\it crossing} of $K$. The images of $0$ and $1$ are the endpoints of $K$. The endpoints are distinct from any of the crossings, and are called \textit{tail} and \textit{head}, respectively. $K$ admits an orientation from the tail to the head. Figure \ref{fig:kntd} shows an example of a knotoid diagram (in $\mathbb{R}^2$ or $S^2$). The trivial knotoid diagram is an embedding of the unit interval $[0,1]$ in $\Sigma$. %We study knotoids in $S^2$ and in $\mathbb{R}^2$ in this paper.

On knotoid diagrams in $S^2$ or in $\mathbb{R}^2$ we allow Reidemeister I, II and III moves each taking place in a local disk free of endpoints, and isotopy of $S^2$ or  $\mathbb{R}^2$, respectively.  
\begin{wrapfigure}{r}{0.3\textwidth}
  \centering
	\includegraphics[width=0.3\textwidth]{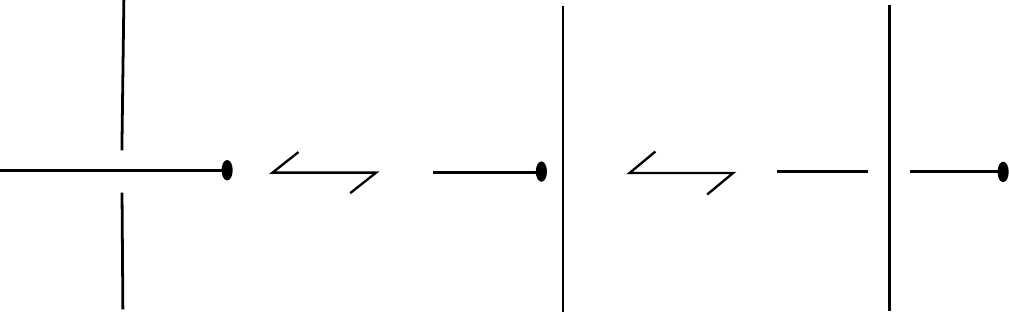}
%\end{center}
\vspace{-12.5pt}
  \hspace{-12.5pt}
\caption{\small Forbidden moves}
\label{subfig:for}
   \vspace{-10pt}
	 \label{fig:forbidden}
\end{wrapfigure}
Pulling/pushing the arc adjacent to an endpoint over or under a transversal arc, as shown in Figure~\ref{subfig:for}, is a \textit{forbidden knotoid move}. Notice that, if forbidden moves were allowed, any knotoid diagram in $S^2$ or in $\mathbb{R}^2$ could be clearly turned into the trivial knotoid diagram. 

A \textit{knotoid} in $S^2$ (or a {\it spherical knotoid}) is an equivalence class of knotoid diagrams in $S^2$ considered up to the equivalence relation generated by the Reidemeister moves and isotopy in $S^2$. A \textit{planar knotoid} is an equivalence class of knotoid diagrams in $\mathbb{R}^2$ taken up to the Reidemeister moves and the isotopy in $\mathbb{R}^2$.

Knotoids in $S^2$ extend classical knot theory \cite{Tu}. Let $\kappa$ be an oriented knot in $S^3$ (or in $S^2 \times I$) and $D$ be an oriented diagram of $\kappa$ lying in $S^2$. Cutting out an open arc from $D$ which is free of crossings results in a knotoid diagram with endpoints in the same local region of the diagram. This induces a well-defined injective map $\alpha$ from the set of oriented knots in $S^3$ to the set of knotoids in $S^2$ \cite{Tu}. The spherical knotoids obtained by this map are called \textit{knot-type knotoids}. Knot-type knotoids admit at least one diagram with its endpoints lying in the same planar region, and they carry the same knottedness information with the knot they are associated with. The spherical knotoids that do not lie in the image of the map $\alpha$ are called \textit{proper knotoids}. % A proper knotoid has non-zero height.

The map $\alpha$ does not form a well-defined map when its image is considered to be the set of knotoids in $\mathbb{R}^2$. This is because unlike classical knots, considering a knotoid diagram in $S^2$ and in $\mathbb{R}^2$ may result in two different knotoids since `Whitney flipping' moves that pull a strand across the north pole are available for spherical knotoid diagrams. The reader can verify that the knotoid diagram in Figure \ref{fig:ex} represents the trivial knotoid in $S^2$ although it can be verified by the loop bracket polynomial (see Section \ref{sec:loop}) that it is non-trivial as a planar knotoid. This implies that one can associate different planar knotoids to the trivial knot through the map $\alpha$.

Another aspect of knotoids in $S^2$ is that they can be considered as simpler representations of classical knots. Two types of classical closures, namely the underpass closure and the overpass closure, are available for a given knotoid diagram. In the underpass closure, two endpoints of a knotoid diagram are connected by an arc which is declared to go under the strands it encounters. This closure operation defines a well-defined and surjective map from the set of knotoids in $S^2$ to the set of oriented classical  knots \cite{Tu}. Note that the underpass closure is clearly a bijection on the set of knot-type knotoids. The knotoids closing to isotopic knots in $S^3$ through the underpass closure can be considered as knotoid representatives of the knots. Knotoid representatives of knots in $S^3$ reduce the computational complexity of knot invariants such as the knot group \cite{Tu} since they may have less crossings than any knot diagrams.
%%%%%%%%%%%%%%%%%%%%%%%%%%%%%%%%%%%%%%%%%%%%%
\subsubsection{Loop bracket polynomial}\label{sec:loop}
In \cite{Tu} Turaev generalizes the Kauffman bracket to a $3$-variable bracket polynomial for knotoids in $\mathbb{R}^2$. In Turaev's $3$-variable bracket polynomial, a variable is associated with the intersection number of the knotoid diagram with an arc chosen to connect the endpoints and also with the intersection number of the bracket states with the chosen arc, and another variable is associated with the number of circular state components enclosing the long state component containing the endpoints.

We restrict this polynomial by setting the variable that is associated to the intersection number of an arc connecting the endpoints with the diagram and with the bracket states equal to $1$. This restricted version of the Turaev polynomial, so called the {\it loop bracket polynomial}, has the skein expansion of the usual bracket, but keeps track of the nesting of loops in the states that surround the long state component. The new variable $O$, is taken as $O^{m}$ when there is a nest of $m$ loops surrounding the long state component in a given state. See Figure \ref{fig:tri}.
\begin{figure}[H]
  \centering
	\includegraphics[width=0.3\textwidth]{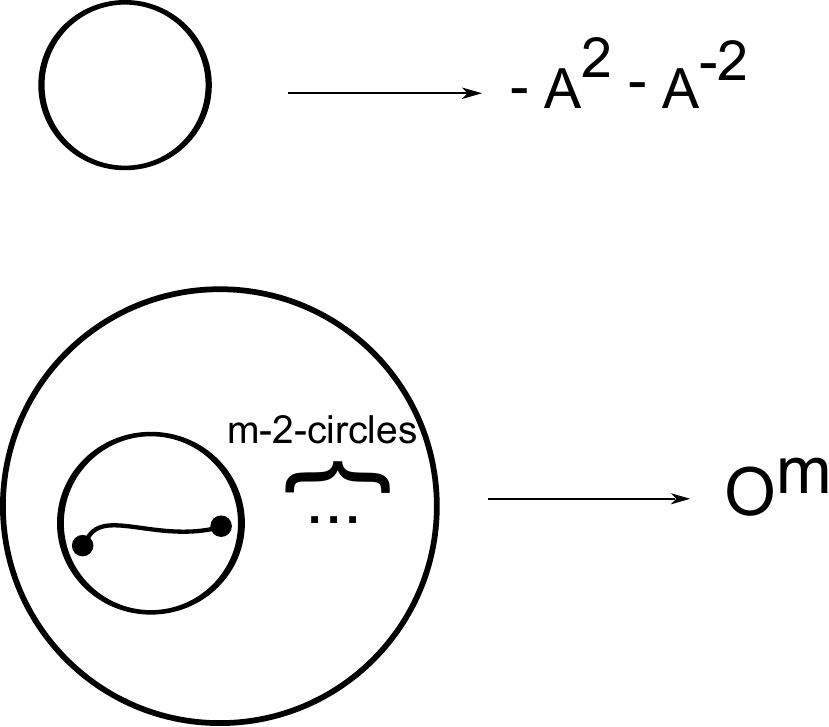}
\caption{\small Circular state components}
	 \label{fig:tri}
\end{figure}
\begin{definition} \cite{GGLDSK} \normalfont
 The \textit{loop bracket polynomial} is defined as follows. 
$$<K>_{O}= (-A^3)^{-wr(K)} \Sigma_{s \in S(K)} A^{\sigma(s)}(-A^2-A^{-2})^{p(s)} O^{q(s)},$$
where $wr(K)$ is the writhe of $K$, $\sigma(s)$ is the number of $A$ smoothing minus the number of $B$ smoothing applied at each crossing of $K$ to obtain the state $s$, $p(s)$ is the number of the circular state components that do not nest around the long state component and $q(s)$ is the number of circular state components nesting around the long state component in the state $s$.
\end{definition}
The loop bracket polynomial is often a stronger invariant than the usual (normalized) bracket polynomial for planar knotoids. The normalized bracket polynomial of the knotoid given in Figure \ref{fig:ex} is trivial but the loop bracket polynomial is nontrivial.   
\begin{figure}[H]
\centering
\includegraphics[width=.55\textwidth]{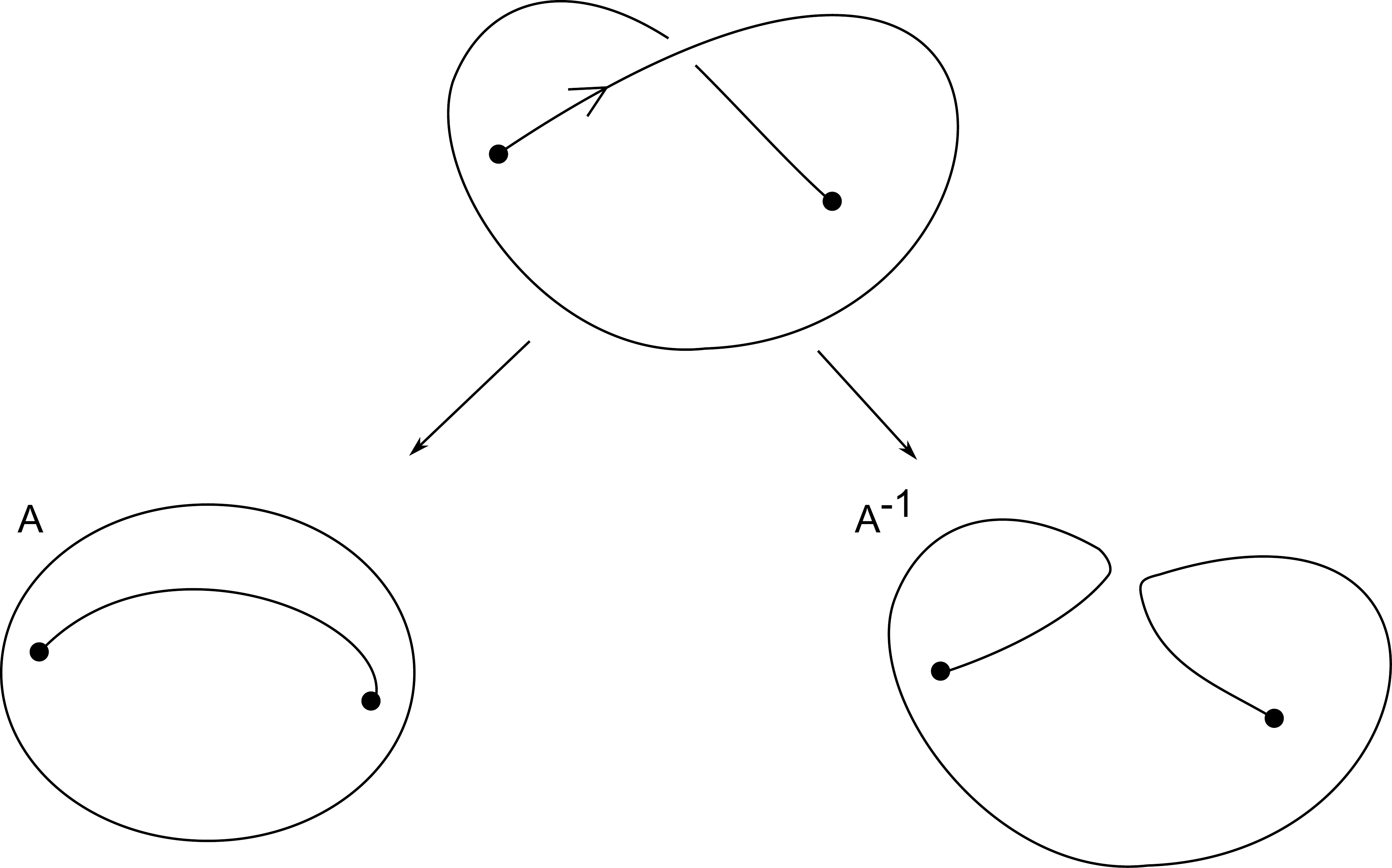}
\caption{The loop polynomial of $K$ is $AO+A^{-1}$}
\label{fig:ex}
\end{figure}
\subsubsection{The arrow polynomials}\label{sec:arrow}
In \cite{GK1} the authors define the arrow polynomial for knotoids in $S^2$ by adapting the arrow polynomial of virtual knots \cite{Ka1}. The arrow polynomial is an oriented generalization of the bracket polynomial of knotoids, defined by assigning new variables to long state components in zig-zag form that are called \textit{irreducible state components}. See Figure \ref{fig:arr}.
\begin{figure}[H]
\centering
\scalebox{1}{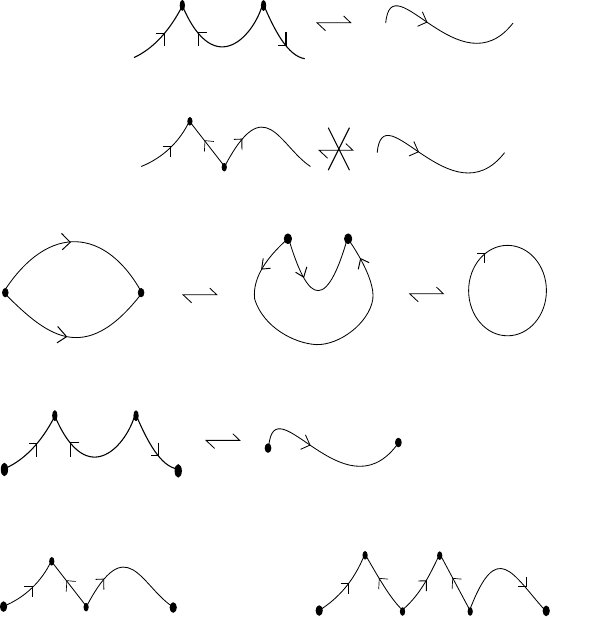}
\caption{}
\label{fig:arr}
\end{figure}
The readers are referred to \cite{GK1} for the definition of the arrow polynomial. Here we note that there exists a chiral version of the arrow polynomial that takes account the chirality of the irreducible state components in $S^2$. The irreducible state components shown in Figure \ref{fig:arrow} are mirror symmetric to each other, and are not isotopic to each other since the the cusps on them point to different local regions in the plane. In the chiral version of the arrow polynomial, the irreducible state components having $n$ zig-zags and are chiral to each other are assigned to different variables $\lambda_n^{+}$  and $\lambda_n^{-}$. 
\begin{figure}[H]
%\vspace{-12pt}
  \centering
	\includegraphics[width=0.3\textwidth]{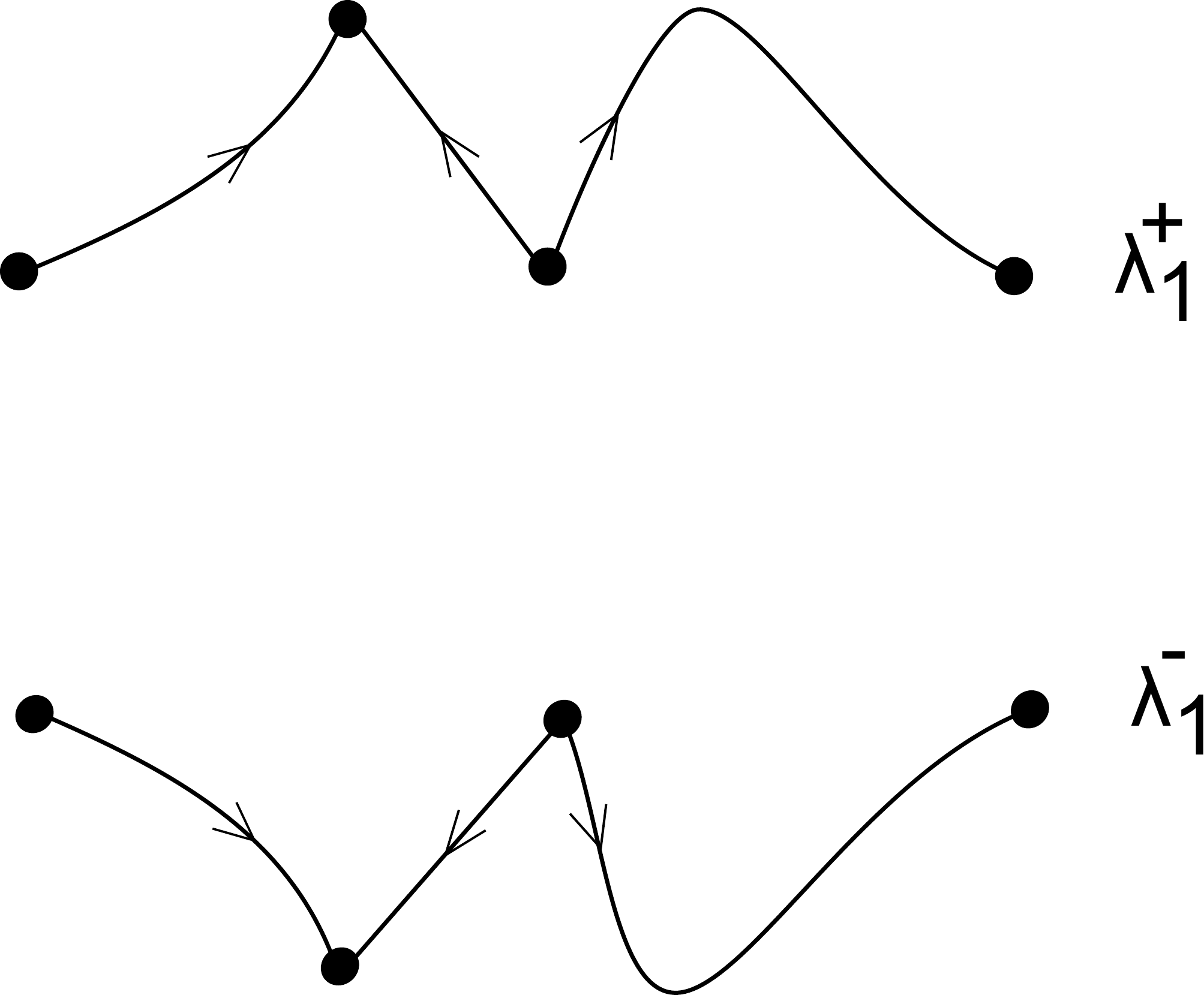}
%\end{center}
%\vspace{-12.5pt}
  %\hspace{-13.5pt}
\caption{\small Irreducible state components}
   \vspace{-10pt}
	 \label{fig:arrow}
\end{figure}
For knotoids in $\mathbb{R}^2$ we can also define a loop arrow polynomial. In the \textit{loop arrow polynomial}, there are three types of variables corresponding to circular state components enclosing long state components. One type of variable is assigned to a long state component without any zig-zags surrounded by a number of circular state components. The other two types of variables are assigned to irreducible state components as shown in Figure \ref{fig:arrow} (with some number of positive or negative zig-zags) that are enclosed by circular state components.
\subsubsection{A complexity invariant of knotoids:The height}
The {\it height} of a knotoid diagram in $S^2$ \cite {Tu} is the least number of crossings created through the underpass closure. The minimum of the heights of knotoid diagrams lying in the same knotoid class is called the {\it height} of the knotoid, and it is an invariant for knotoids in $S^2$ \cite{Tu}. The height of a knotoid is an intrinsic invariant determining the type of a knotoid. Clearly, a knotoid in $S^2$ has height zero if and only if it is a knot-type knotoid.  In \cite{GK1} the authors show that the \textit{affine index polynomial} and the \textit{arrow polynomial} defined for knotoids in $S^2$ provide lower bound estimations for the height of a knotoid. 
%%%%%%%%%%%%%%%%%%%%%%%%%%%%%%%%%%%%%
%%%%%%%%%%%%%%%%%%%%%%%%%%%%%%%%%%%
\subsection{Basics on virtual knots}
 The theory of virtual knots was introduced by the second listed author \cite{Ka1,Ka2}. A \textit{virtual knot} is an embedding of the unit circle in a thickened surface of arbitrary genus modulo isotopies and orientation preserving diffeomorphisms plus one-handle stabilization/destabilization of thickened surfaces. 
A virtual knot may be represented by different embeddings in different thickened surfaces. We call a representation of a virtual knot a \textit{minimal representation} if it is an embedding of the knot that does not admit destabilization. %In other words, it is a representation of a virtual knot in a thickened surface 
\begin{theorem}(Kuperberg)\cite{Ku}
A minimal representation of a virtual knot is unique up orientation preserving homeomorphisms.
\end{theorem}
\begin{definition}\normalfont
The \textit{virtual genus} of a virtual knot is the genus of the surface of its minimal representation.
\end{definition}
Classical knots (knots in $S^3$ or in $S^2 \times I$) can be considered as virtual knots of (virtual) genus $0$. From Kuperberg theorem it follows that classical knot theory properly embeds into virtual knot theory, that is, if two classical knots are virtually equivalent then they are also equivalent to each other by classical Reidemeister moves. 

The theory of virtual knots has also a diagrammatic formulation: A virtual knot in a thickened surface of some genus can be represented by a \textit{virtual knot diagram} in $S^2$ or in $\mathbb{R}^2$ that contains a finite number of \textit{real/classical crossings}, 
\begin{wrapfigure}{r}{0.4\textwidth}
\vspace{-10pt}
  \centering
	\includegraphics[width=0.4\textwidth]{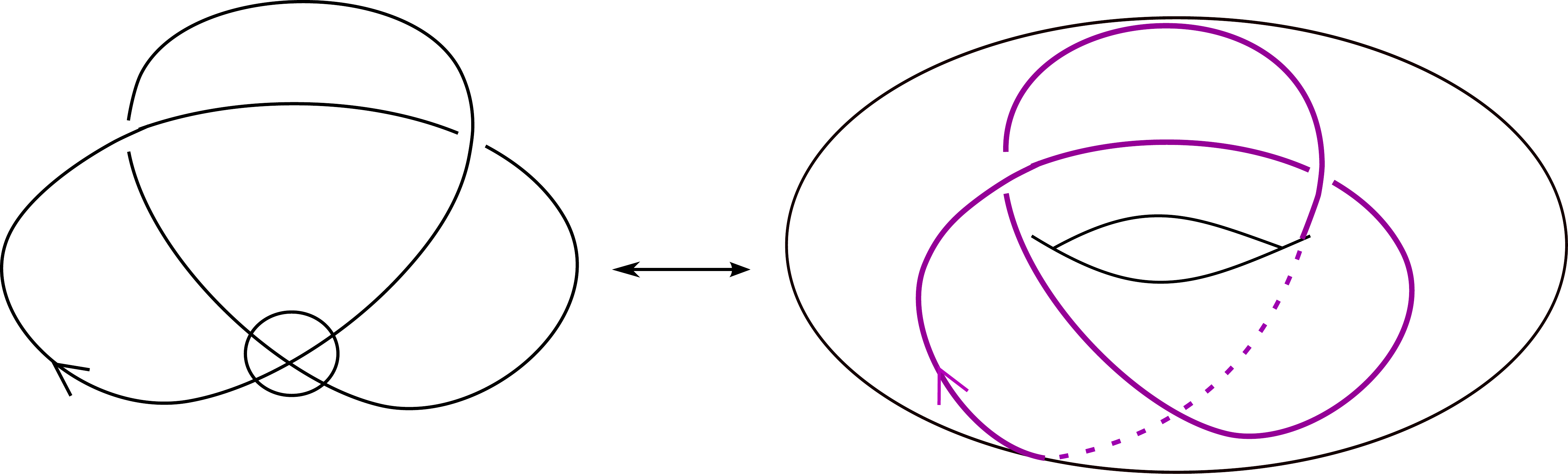}
%\end{center}
  \hspace{-12.5pt}
\caption{\small A virtual knot diagram
}
   \vspace{-10pt}
	 \label{fig:virt}
\end{wrapfigure}
and \textit{virtual crossings}. A virtual crossing is indicated by a small circle placed around a crossing point as shown in Figure \ref{fig:virt}. The arcs containing a virtual crossing correspond to arcs of the virtual knot one of which lies at the front of a handle and the other lies at the back of the same handle of the thickened surface. 
\begin{wrapfigure}{r}{0.35\textwidth}
  \centering
	\vspace{-13pt}
	\includegraphics[width=0.35\textwidth]{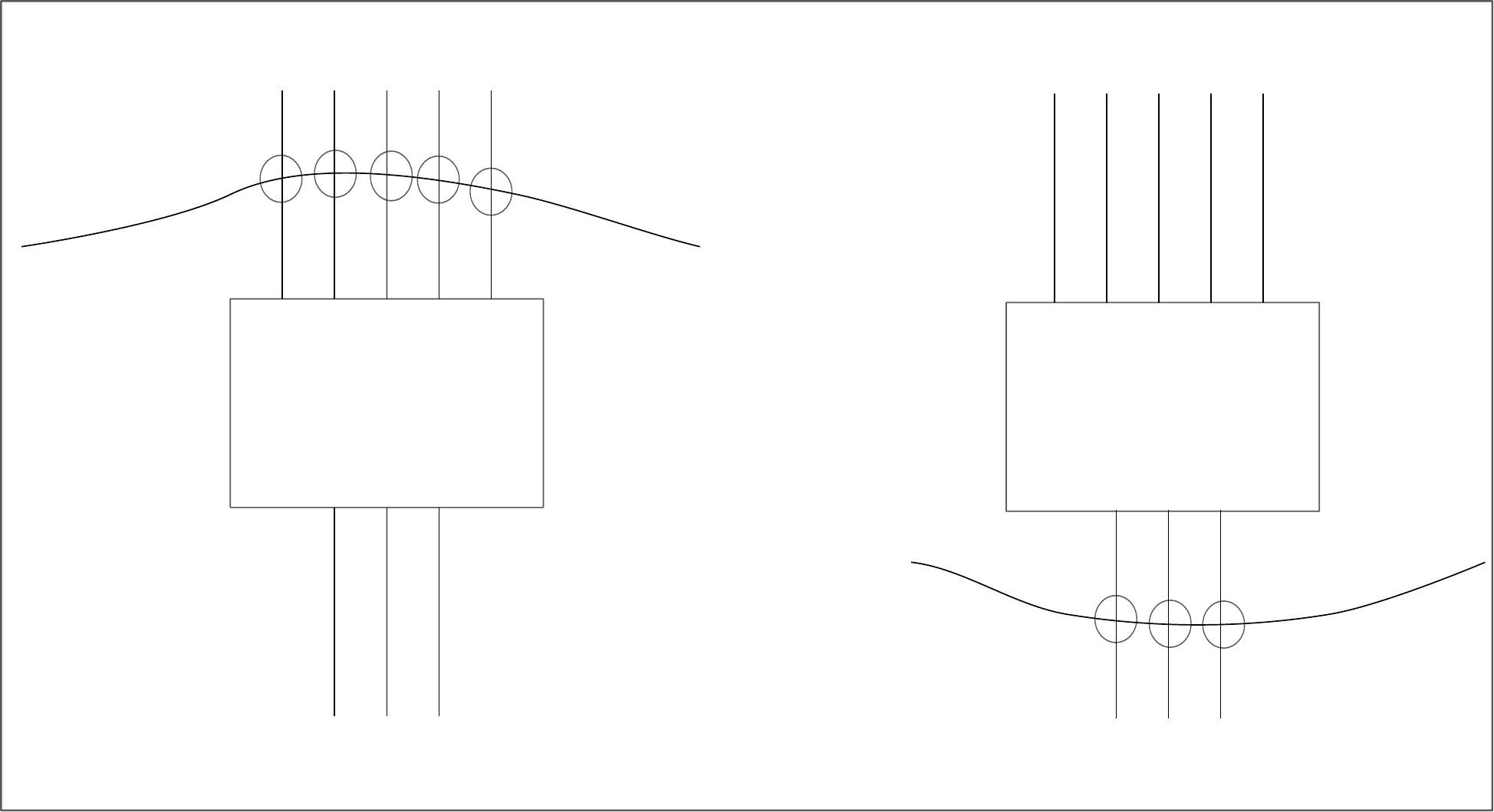}
%\end{center}
\vspace{-12.5pt}
  \hspace{-12.5pt}
\caption{\small The detour move
}
   \vspace{-10pt}
	 \label{fig:det}
\end{wrapfigure}
The moves on virtual knot diagrams are generated by the usual Reidemeister moves plus the detour move. The \textit{detour move} allows a segment with consecutive sequence of virtual crossings to be excised and replaced by any other such a segment with a consecutive virtual crossings, as shown in Figure \ref{fig:det}. Two virtual knot diagrams are \textit{virtually equivalent} if they can be related to each other by a finite sequence of the Reidemeister and detour moves. A \textit{virtual knot} is a virtual equivalence class of virtual knot diagrams.

Theorem \ref{thm:vkt} provides a one-to-one correspondence between the topological and the diagrammatic approaches of virtual knot theory. The transition between the two approaches is facilitated by abstract knot diagrams. An abstract knot diagram associated to a virtual knot diagram is a ribbon-neighborhood surface containing the knot diagram. The abstract knot diagram surface is an oriented surface with boundary. In Figure \ref{fig:abstract} we illustrate the abstract knot diagram associated with the virtual knot diagram given in Figure \ref{fig:virt}. For more details on abstract knot diagrams the reader is referred to \cite{KK,CKS}. 
\begin{figure}[H]
\centering
\includegraphics[width=.7\textwidth]{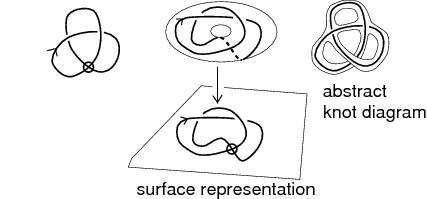}
\caption{\small Virtual knot representations}
\label{fig:abstract}
\end{figure}
 \begin{theorem} (\cite{Ka1, Ka3, CKS})\label{thm:vkt}
Two virtual knot diagrams are virtually equivalent if and only if their surface embeddings are equivalent up to isotopy in the thickened surfaces, diffeomorphisms of the surfaces, and addition/removal of empty handles.
\end{theorem}
Another approach to study virtual knot theory is via oriented Gauss diagrams \cite{Go}. An oriented Gauss diagram is an oriented circle with a number of chords each connecting a pair of points on the circle. Each pair of points connected by a chord corresponds to a real crossing and the chord encodes the information of real crossing with an orientation on it. The equivalence on Gauss diagrams is generated by the abstract Reidemeister moves that are analogues of Reidemeister moves on Gauss diagrams.      
\begin{theorem}\cite{Go,Ka1}
Two virtual knot diagrams are virtually equivalent if and only if their corresponding Gauss diagrams are related to each other by abstract Reidemeister moves.
\end{theorem}
The notions of virtual knot theory including parity apply to knotoids naturally. Parity in virtual knot diagrams and related invariants are examined in \cite{Ma,Ma3,Kauf}. Manturov studied \cite{Ma} a parity projection map from virtual knots to classical knots. We discuss more on this map in Section \ref{sec:minimal}. 
%%%%%%%%%%%%%%%%%%%%%%%%%%%%%%%%%%%%%%%%%%%55
\section{Parity in knotoids}\label{sec:parity}
\subsection{Gauss codes and the Gaussian parity}
  A \textit{Gauss code} is a linear code that consists of a sequence of labels assigned to crossings of a knotoid diagram in $S^2$
	\begin{wrapfigure}{r}{0.22\textwidth}
  \centering
	\vspace{-12pt}
	\includegraphics[width=0.14\textwidth]{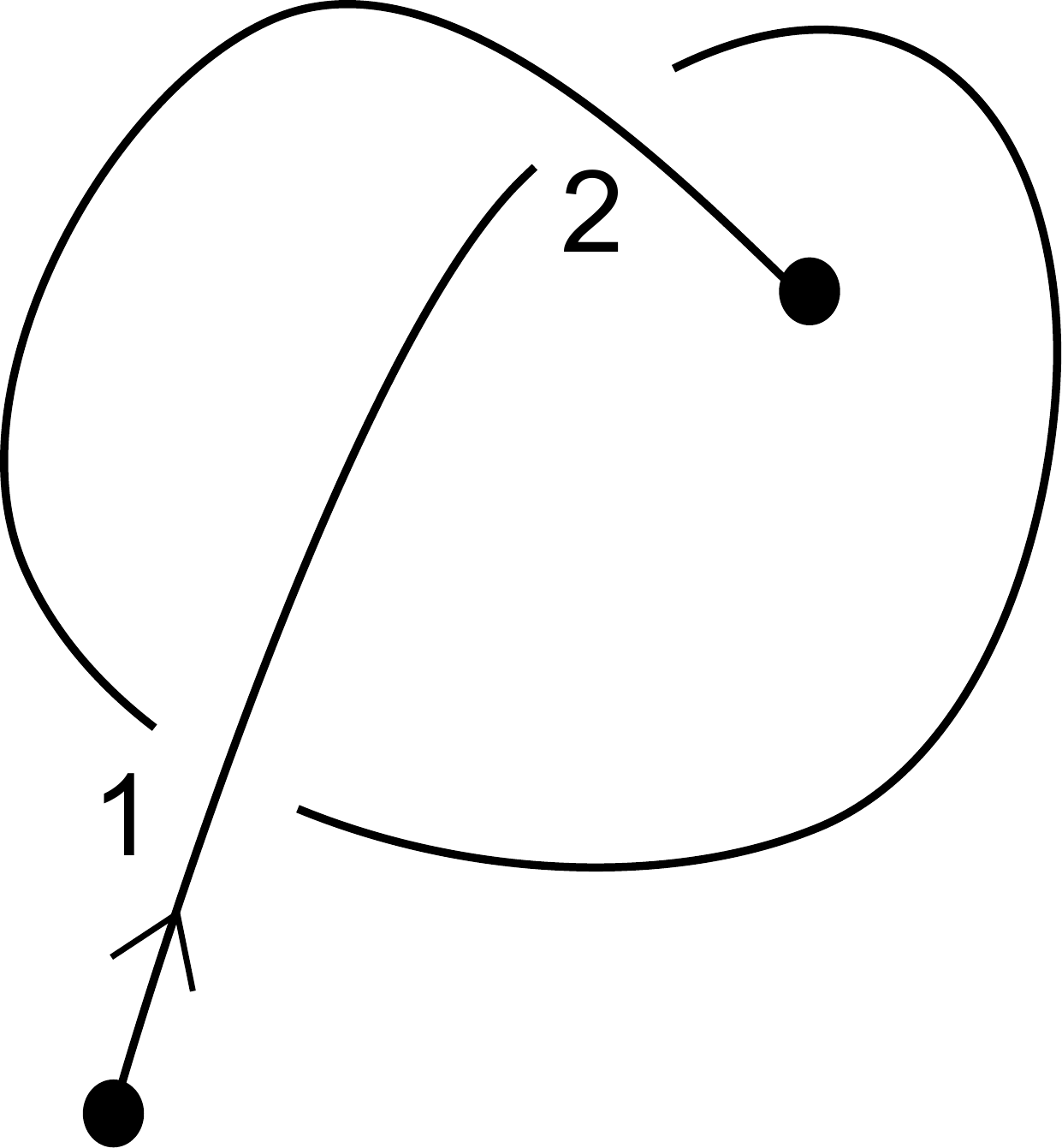}
%\end{center}
  \vspace{-13pt}
\caption{\small The Gauss code is $+O1+U2U1O2$}
   \vspace{-10pt}
	 \label{fig:gauss}
\end{wrapfigure}
or in $\mathbb{R}^2$. It encodes crossings that are encountered during a trip along the knotoid diagram following the orientation. 
Since any crossing is traversed twice, each label in the code appears twice. This gives a code of length $2n$, where $n$ is the number of crossings. We keep the information of the passage through a crossing either as an overcrossing or an undercrossing by adding the symbols $O$ and $U$, respectively, and the signs of the crossings by putting $+$ or $-$ next to the label. Clearly, each knotoid diagram has a unique Gauss code. See Figure \ref{fig:gauss} for an example. 
	
A Gauss code of a knotoid diagram is said to be \textit{evenly intersticed} if there is an even number of labels between two appearances of any label. Similarly for classical knots \cite{Ka1}, it is shown in \cite{GK1} that the Gauss code of a knotoid diagram in $S^2$ is evenly intersticed if and only if it is a knot-type knotoid diagram. Equivalently, the Gauss code of a proper knotoid diagram has a non-evenly intersticed Gauss code \cite{GK1}. This fact gives rise to a well-defined parity map that is defined on the set of crossings in knotoid diagrams and takes values in $\mathbb{Z}_2$. A crossing of a knotoid diagram is assigned parity  $1$ and called \textit{odd} if there are an odd number of labels in between the two appearances of the crossing, otherwise it is assigned to parity $0$ and called an \textit{even} crossing. 

Using this parity of crossings, we define an invariant for knotoids called the \textit{odd writhe} \cite{GK1}.
\begin{definition}\normalfont
 {\it Odd writhe} of a knotoid diagram $K$ in $S^2$, $J(K)$ is the sum of the signs of the odd crossings,
\begin{center}
$J(K) = \sum_{c\in Odd(K)}sign(c)$,
\end{center}
where $Odd(K)$ is the set of odd crossings in $K$.
\end{definition}
The knotoid represented by the knotoid diagram in Figure \ref{fig:gauss} has odd writhe equal to $2$.

One proves that $J(K)$ is invariant by observing that the crossing at an Reidemeister I move is always even, the two crossings at an Reidemeister II move are either both even or both odd, and at an Reidemeister III move either two crossings are odd  or all three crossings are even. These properties of Gaussian parity are axiomatized by Manturov in \cite{Ma} and form the basis for other parity invariants of virtual knots.
\vspace{-2cm}
\subsection{Parity bracket polynomial of knotoids}
 The parity bracket polynomial of V.~Manturov \cite{Ma} is a modification of the bracket polynomial that uses the parity of crossings in virtual knots. With the existence of even and odd crossings in knotoid diagrams, the parity bracket polynomial can be defined for knotoid diagrams. 
%\subsubsection{Parity bracket polynomial for spherical knotoids}

For a knotoid diagram $K$ in $S^2$ or in $\mathbb{R}^2$, a \textit{parity state} is defined to be a labeled graph that is obtained by first replacing each odd crossing of $K$ with a graphical node and then the remaining crossings are expanded by the bracket relation shown at the top of Figure \ref{fig:skein} without any further check about odd or even. The resulting summation over crossing-free diagrams is identical to the state sum once the states are made irreducible by the reduction rule in the Figure \ref{fig:skein}.% smoothing the even crossings of $K$ by A and B type of smoothing of the usual bracket polynomial, see Figure \ref{fig:skein}, and by labeling the smoothing sites by $A$ and $A^{-1}$, respectively.% Odd crossings of $K$ are replaced by graphical nodes. %We regard the circular and long segment components of a parity state as planar graphs.
\begin{figure}[H]
%\centering  \scalebox{0.5}{\input{trl.pdf_tex}}
 %   \begin{center}
  %   \begin{tabular}{c}
     \centering  \scalebox{0.75}{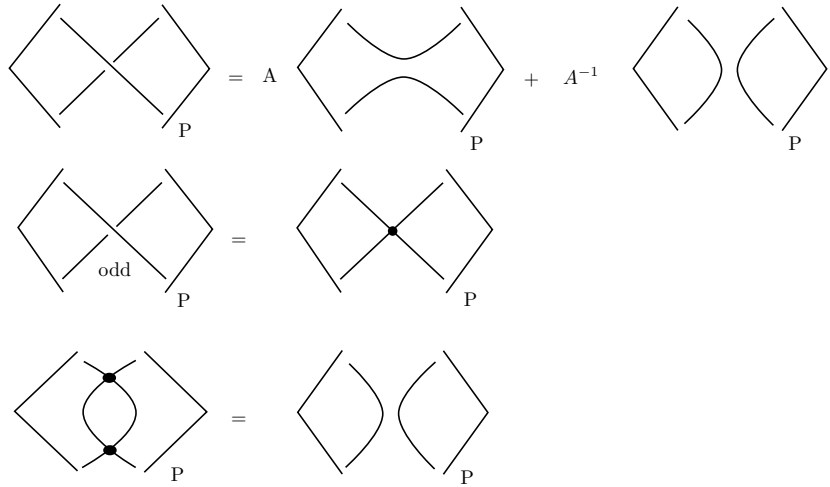}
   %  \end{tabular}
     \caption{Parity bracket expansion}
     \label{fig:skein}
%\end{center}
\end{figure}	
The reduction rule is simply a type of Reidemeister II move that eliminates two graphical nodes forming the vertices of a bigon. Each state component that contains nodes after applying all possible reductions  contributes to the polynomial as a graphical coefficient. Circular or long segment state components without any nodes left, contribute as the value of $-A^2- A^{-2}$. % formed by two nodes corresponding to odd crossings of a knotoid diagram in $S^2$. 
 %The skein expansion in Figure \ref{fig:skein} can be interpreted as follows. 
\begin{definition}\normalfont
The \textit{parity bracket polynomial} of a knotoid diagram $K$ in $S^2$ or in $\mathbb{R}^2$, is defined as
 \begin{center}
$<K>_P = \sum_S A^{n(S)}(-A^2-A^{-2})^{l(S)}G(S)$,
\end{center}
where $n(S)$ denotes the number of $A$-smoothings minus the number of $B$-smoothings, $l(S)$ is the number of components without any nodes, and $G(S)$ is the union of state components containing nodes in the state $S$. 
\end{definition}
The \textit{normalized parity bracket polynomial} of a knotoid diagram $K$ in $S^2$ or in $\mathbb{R}^2$ is defined as
\begin{center}
$P_K = (-A^3)^{-wr(K)}<K>_P$,
\end{center}
where $wr$ is the writhe of $K$.
The normalized parity bracket polynomial is an invariant of planar and spherical knotoids \cite{GK1}. 

 It is shown \cite{GK1} that the isotopy of $S^2$ and the reduction rule eliminate any bigon region. As a result, there is no graphical coefficient in the parity bracket polynomial of a knotoid in $S^2$ \cite{GK1}. As we shall see below, this elimination does not neccessarily happen for knotoids in $\mathbb{R}^2$.
\subsubsection{Parity bracket polynomial of planar knotoids}
For planar knotoids, the parity states are then taken up to planar isotopy rather than isotopy of $S^2$. This makes it possible for some graphical coefficients to survive.
We give an example of two planar knotoids $K$ and $K'$ in Figure \ref{fig:kk} that are distinguished in this way. The knotoid $K$ is listed as the knotoid $2_2$ in the table \cite{Dim} up to an orientation, and $K'$ is the symmetric \cite{Tu} of $K$.  %For our purposes in this note it suffices to just show how they are distinguished by the planar parity bracket.
Both crossings of the planar knotoids $K$ and $K'$ are odd crossings. The planar isotopy classes of the resulting graphs $G$ and $G'$ are the planar parity bracket invariants of 
$K$ and $K'$, respectively. We leave it as an (easy) exercise for the reader to verify that $G$ and $G'$ are not planar isotopic and therefore $K$ and $K'$ are not equivalent planar knotoids. 
\vspace{-.53cm}
\begin{figure}[H]
\centering \includegraphics[width=.47\textwidth]{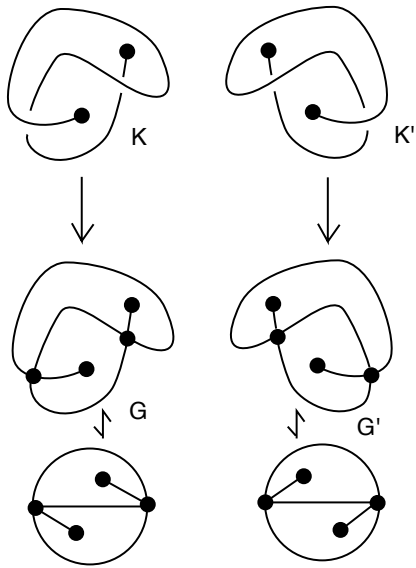}
\caption{Distinguishing symmetric knotoids in the plane}
\label{fig:kk}
\end{figure}
\begin{figure}[H]
\centering \includegraphics[width=.47\textwidth]{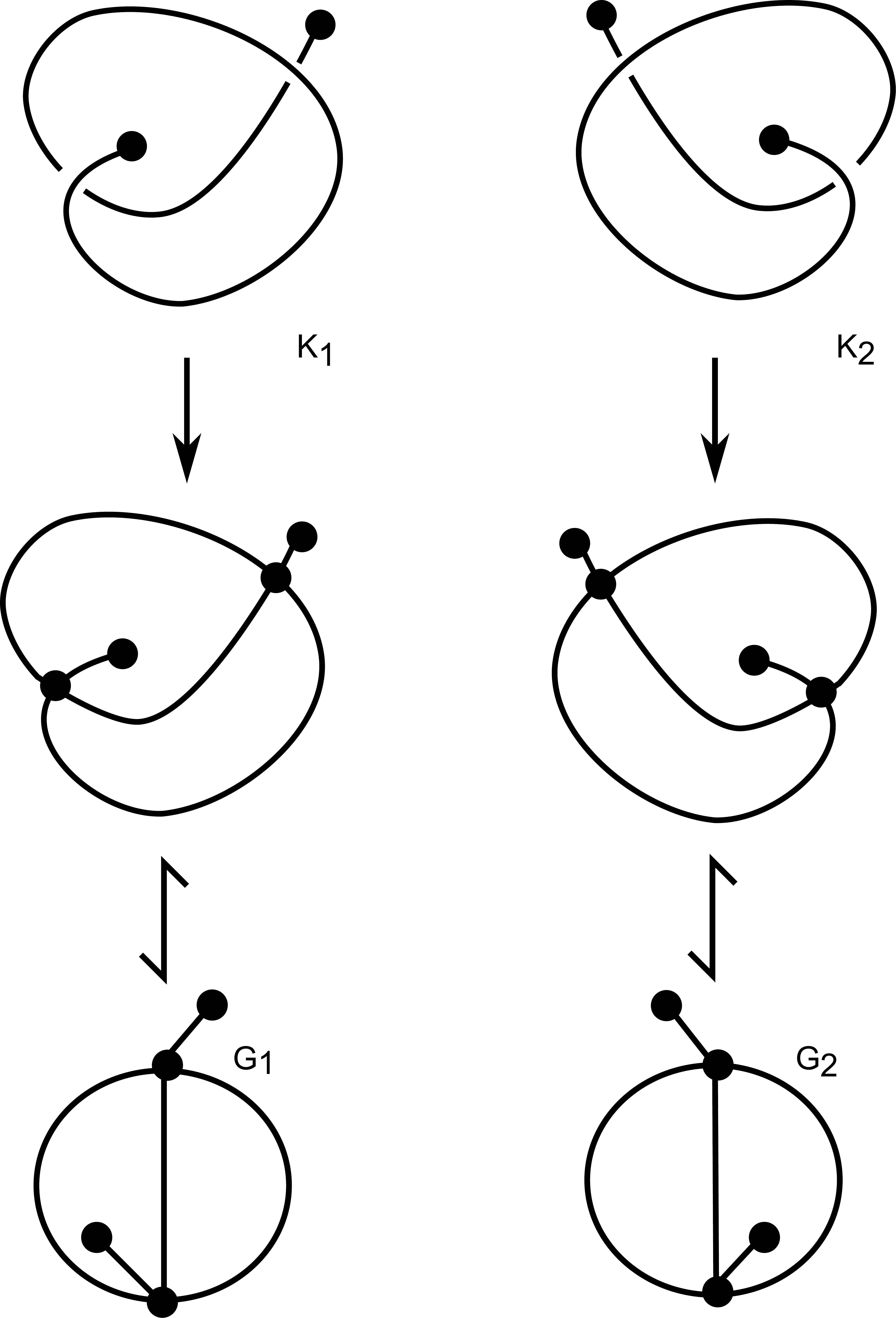}
\caption{}
\label{fig:partwo}
\end{figure}
A further calculation shows that these pairs have the same loop bracket polynomial but are distinguished by the loop arrow polynomial discussed in Section \ref{sec:arrow} when given an orientation.
%%%%%%%%%%%%%%%%%%%%%%%%%%%%%%%%%%%%%%%%5
\section{A connection between knotoids and virtual knots} \label{sec:virtclosure}
 The endpoints of a knotoid diagram $K$ in $S^2$ can be connected with an arc by encoding each intersection of the connecting arc
\begin{wrapfigure}{r}{0.25\textwidth}
  \centering
	\includegraphics[width=0.25\textwidth]{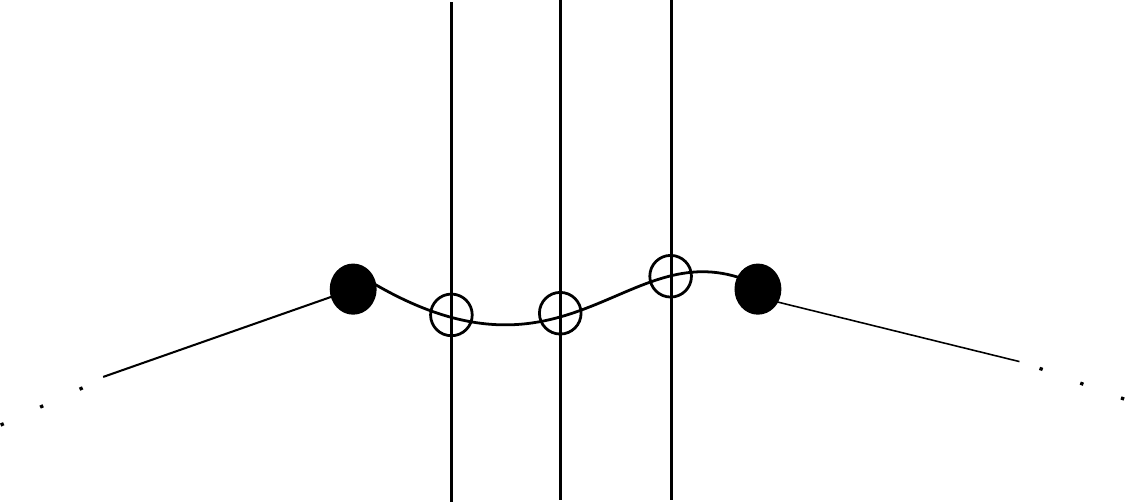}
%\end{center}
  \vspace{-13.5pt}
\caption{\small The virtual closure}
   \vspace{-10pt}
	 \label{fig:vir}
\end{wrapfigure}
with a strand of the diagram as a virtual crossing, see Figure \ref{fig:vir}. The resulting diagram is clearly a virtual knot diagram that can be also represented in a torus. Precisely, two disks around the endpoints of $K$ are cut out from the $2$-sphere of $K$ and a $1$-handle is attached to the resulting annulus by identifying the boundaries via orientation reversing homeomorphisms. The attached handle is assumed to hold the arc connecting the endpoints of $K$. The resulting torus representation of the knot, exemplified in Figure \ref{fig:stnd}, is called the \textit{standard torus representation}.
\begin{figure}[H]
\includegraphics[width=.5\textwidth]{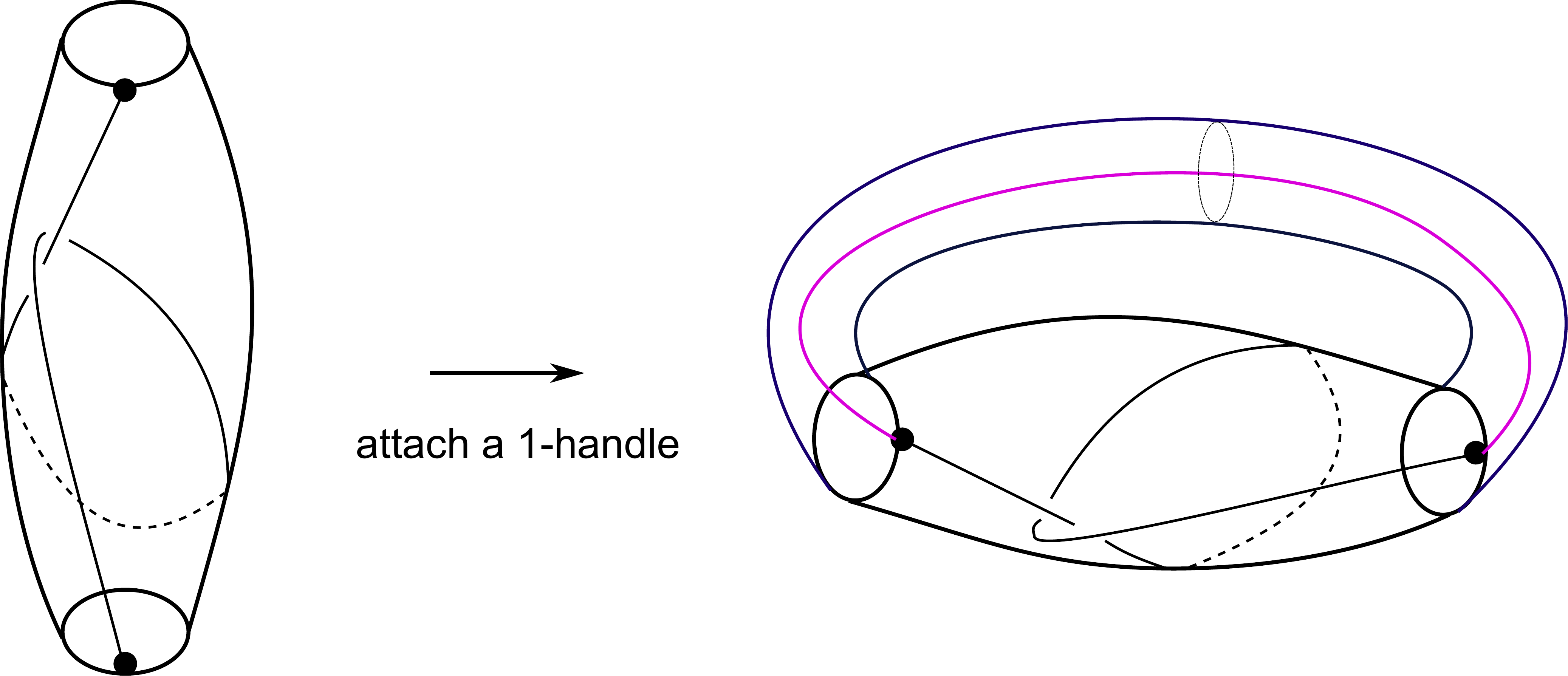}
\caption{Standard closure of a knotoid}
\label{fig:stnd}
\end{figure}
In this way we obtain a well-defined map \cite{Tu,GK1},
\begin{center}
$\overline{v}$: \{Knotoids in $S^2$ \} $\rightarrow$ \{Virtual knots of virtual genus $\leq$ $1$\},
\end{center} 
that is called the \textit{virtual closure map}. Since the map $\overline{v}$ is well-defined, any invariant of virtual knots induces an invariant for knotoids through the virtual closure map. The knotoid invariants obtained by the virtual closure map are stronger than the invariants obtained by classical closures of knotoids since adding virtual crossings does not change the knottedness information of the knotoid \cite{GK1}. 

Clearly the virtual closure of a knot-type knotoid is a classical knot. It follows from Korablev and May \cite{Ko} that the virtual closure will produce a genus $0$ knot only if the knotoid is of knot-type, that is, only if the knotoid has height $0$. The representation of the virtual closure of a knotoid in a thickened torus is the same as Korablev and May's `lifting' of the knotoid. A knot in a thickened torus destabilizes if it can, in our terms, be represented in genus $0$.
\begin{theorem}\cite{Ko}\label{thm:Ko}
A knotoid has height $0$ if and only if its lifting admits destabilization.
\end{theorem}
From Theorem \ref{thm:Ko} and the fact that knot-type knotoids (knotoids with height $0$) carry the same topological information as their closures (virtual or classical) we can conclude that there is no proper knotoid that closes virtually to the trivial knot and in fact, only the trivial knotoid closes virtually to the trivial knot. 

It is not hard to show that the virtual closure map is not an injective map. The pair of knotoids given in Figure \ref{fig:nonin} has the same virtual closure. It can be easily verified that the knotoid on the left hand-side is $3$-colorable but the knotoid on the right hand-side is mono-chromatic, so they are non-equivalent knotoids.
\begin{figure}[H]
\includegraphics[width=.5\textwidth]{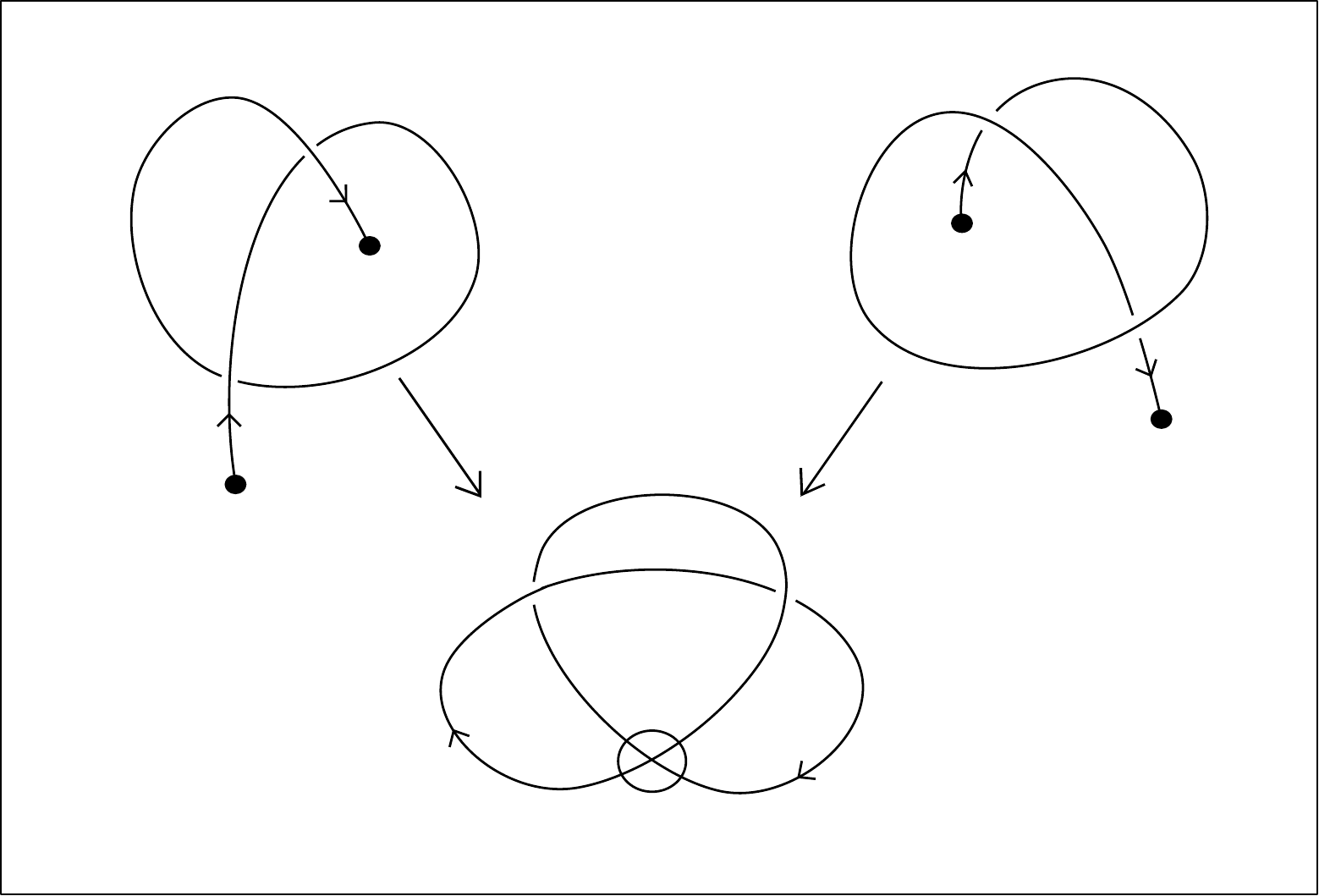}
\caption{A pair of non-equivalent knotoids with the same virtual closure}
\label{fig:nonin}
\end{figure}
\subsubsection{The virtual closure map is not surjective}\label{sec:surj}
We show that the virtual closure map is not surjective by examining the \textit{surface bracket states}. Let $(\Sigma, K)$ be a fixed surface representation of the virtual knot $\kappa$. We apply Kauffman bracket state expansion by smoothing out all crossings of $K$ in $A$ and $B$ type. The resulting collection of disjoint simple closed curves in $\Sigma$ are the \textit{surface bracket states}.
\begin{definition}  \cite{DK1} \normalfont
 The {\it surface bracket polynomial} of $(\Sigma, K)$ is defined as,
%  state components that bound a disk in $\Sigma$ contributes to the polynomial that do not bound a disk 
$$<(\Sigma, K)> = \sum_{(\Sigma, s)} A^{\sigma(s)} d^{|s(c)|} [s(c)],$$
where $\sigma(s)$ is the number of $A$ smoothings minus the number of $B$ smoothings, $|s(c)|$ is the number of components of the state $s$ that bound a disk in $\Sigma$, and $[s(c)]$ is the formal sum of homology classes of components of the state $s$ that do not bound a disk.  
\end{definition}
%Dye and the second listed author \cite{} proved the following minimality theorem for virtual knots by using the surface bracket polynomial. 
%\begin{theorem}
%Let $(\Sigma, K)$ be a surface representation of a virtual knot diagram $K$, and $\Sigma=T_{1} T_{2} T_{3} T_{4}$
%\end{theorem}
\begin{definition} \normalfont
A surface bracket state has \textit{multiplicity one} if it has only one component that does not bound a disk.
\end{definition}
In Figure \ref{fig:surface} we verify that each surface bracket state of the standard representation of the virtual closure of the knotoid $K$ given in Figure \ref{fig:surface} has multiplicity one.
\begin{figure}[H]
\centering
\includegraphics[width=1\textwidth]{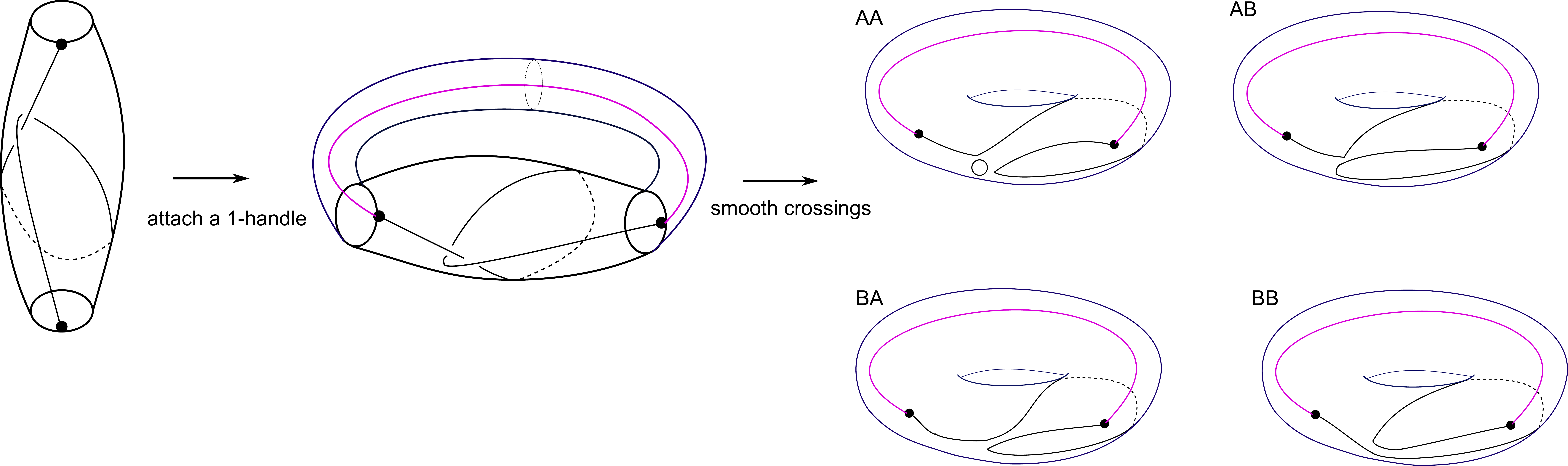}
\caption{Surface bracket states of the standard representation of $\overline{v}(K)$}
\label{fig:surface}
\end{figure}
\begin{lemma}\label{lem:mone}
Given any standard torus representation of a virtual knot lying in the image of the virtual closure map.
Each homologically nontrivial surface bracket state of such knot diagram in the torus has multiplicity one. In fact they can be only of the form $[\lambda] + k[\mu]$, where $[\lambda], [\mu]$ are the longitudinal and the meridional generators of $H_1(T^2)$, respectively, and $k \in \mathbb{Z}$. 
\end{lemma}
\begin{proof}
In a standard representation of a virtual knot lying in the image of the virtual closure map, there is an arc attached to the ends of the knotoid that goes around the handle added to the $2$-sphere minus two disks neighboring the endpoints. This is the closure arc. The endpoints of the knotoid are on the two boundary components of the annulus. In any surface state of the closure the two endpoints are connected by the long segment component. With the closure arc that goes around the attached handle, the long segment component forms a non-bounding component in the torus which is of the form $[\lambda] + k[\mu]$, where $k \in \mathbb{Z}$. Furthermore, the long segment component cuts the annulus into two disks by the Jordan curve theorem. As a result all remaining components bound disks in the annulus and hence in the torus. 
\end{proof}
%\vspace{-2cm}
\begin{theorem}\label{thm:virtual}
Let $\kappa$ be a virtual knot that lies in the image of the virtual closure map, and $(T, K)$ be any torus representation of $\kappa$. Then each homologically nontrivial surface bracket state of $(T, K)$ has multiplicity one.
\end{theorem}
\begin{proof}
By Lemma \ref{lem:mone}, any surface bracket state of a standard representation of $\kappa$ has multiplicity one. This property is true for any torus representation of the knot $\kappa$ since $\kappa$ has virtual genus one and minimal representations of $\kappa$ are related to each other by an orientation preserving homeomorphism of the torus by the Kuperberg theorem. 
 \end{proof}
\begin{corollary}
Let $\kappa$ be a virtual knot of genus one and $(T^2, k)$ denote a torus representation of $\kappa$. As a corollary of Theorem \ref{thm:virtual}, if at least one of the nontrivial surface state components of $(T^2, k)$ is of the form  (for some choice of orientation) $a[\lambda]$ or $b[\mu]$ for some $a, b \in \mathbb{Z}-\{0\}$, $|a| , |b| \neq 1$, that is, it is of some non-trivial multiplicity, then $k$ is not in the image of $\overline{v}$.
\end{corollary}
Now we can verify that the virtual knot with a torus representation given in Figure \ref{fig:stt}, is not the virtual closure of any knotoid in $S^2$. We observe that the isotopy classes of the homologically non-trivial surface state curves are of the form (or homologous to) $[2a]$ and $[2b]$, up to a choice of orientation, meaning that, they have non-trivial multiplicity. Therefore, by Theorem \ref{thm:virtual} we deduce that $K$ does not lie in the image of the virtual closure map.
\begin{figure}[H]
\centering \includegraphics[width=.85\textwidth]{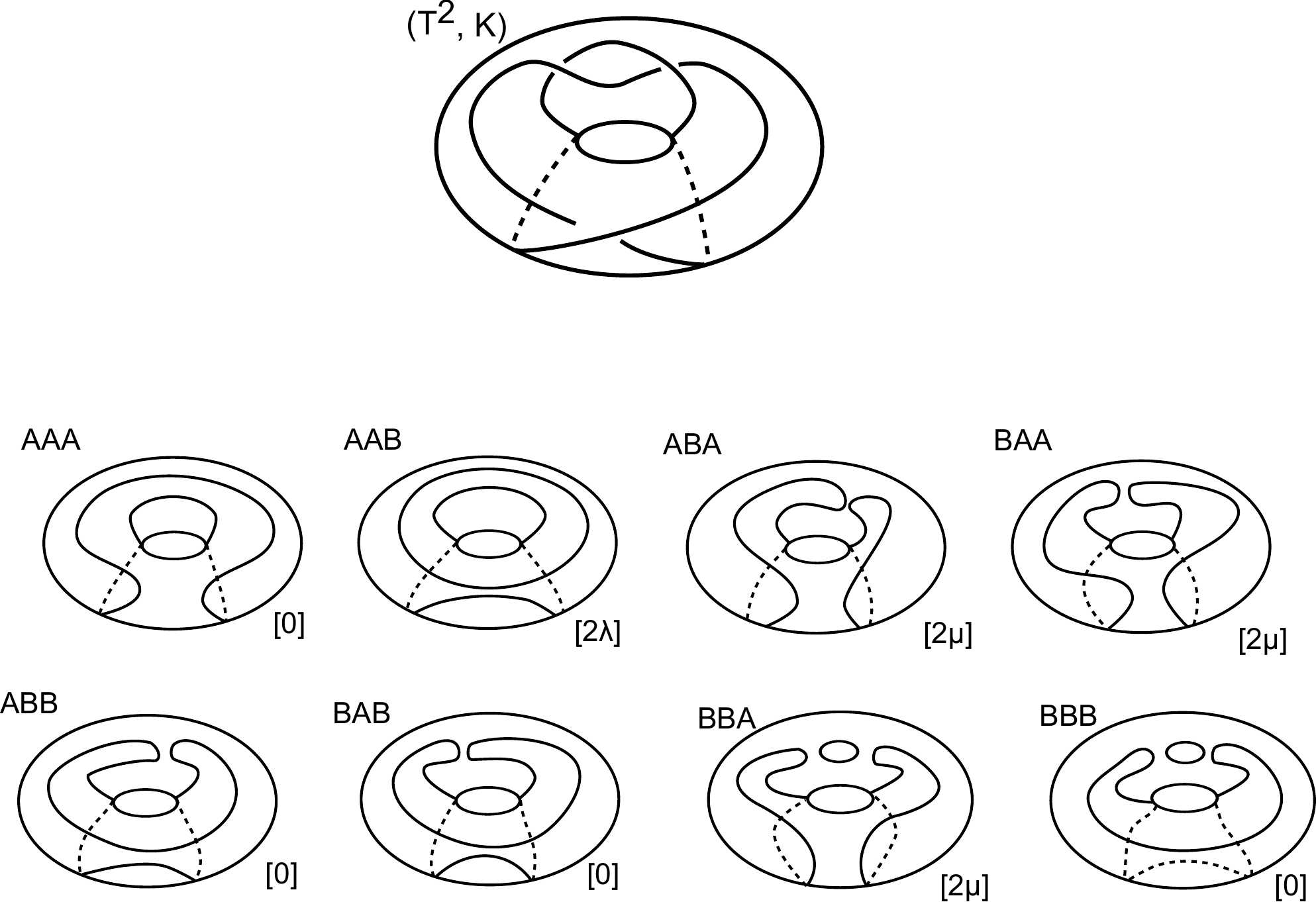}%
\caption{Surface state curves of $K$}
\label{fig:stt}
\end{figure}
%\vspace{-2cm}
\begin{definition}\normalfont
Let $K$ be a classical knot diagram. {\it Virtualization} of $K$ at a single (classical) crossing is defined to remove a $(2,2)$-tangle containing a crossing and replace it with a virtual $(2,2)$-tangle that contains the crossing switched and flanked with two virtual crossings. See Figure \ref{fig:virtualization}.
\end{definition}
\begin{figure}[H]
 % \centering \scalebox{0.4}	{\input{virtlz.pdf_tex}}
\centering \includegraphics[width=.6\textwidth]{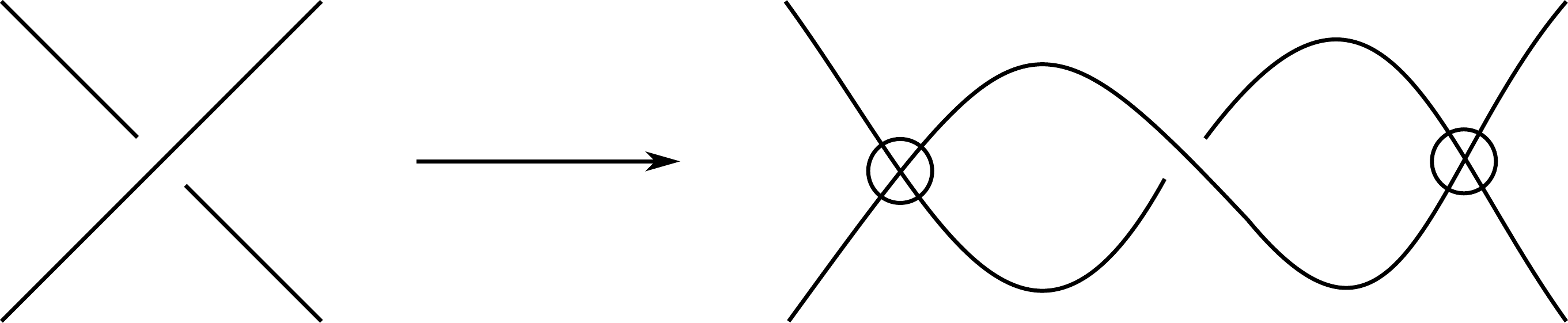}
\caption{Virtualization at a single crossing}
\label{fig:virtualization}
\end{figure}
\vspace{-0.47cm}
 The resulting diagram $K_v$ is a virtual knot diagram that can be represented also in a torus \cite{Ka1}. Notice that the virtual knot diagram given in Figure \ref{fig:stt} can be regarded as the virtualization of the trefoil at a single crossing.

 The second listed author showed \cite{Ka1} that there is an infinite collection of virtual knots whose Jones polynomial are trivial. This collection is obtained by virtualizing a sequence of unknotting crossings. For some time it was an open problem in virtual knot theory where a virtual knot obtained in this way could be virtually equivalent to a classical knot. In the paper \cite{DKK} it was shown, 
by using Khovanov homology for virtual knots, that no virtual knots obtained in this fashion are classical. Silver and Williams proved the following theorem for a specific class of virtualized knots by using a special feature of classical knot diagrams observed in the Wirtinger presentation of the knot group \cite{SiWi}.
\begin{theorem}(Silver and Williams)\cite{SiWi}\label{thm:Si}
Let $K$ be a non-trivial classical knot diagram and $v$ is an unknotting crossing. If $K_v$ is the virtual knot obtained by virtualizing the crossing $v$ then $K_v$ is a non-classical and non-trivial virtual knot. 
\end{theorem}
Theorem \ref{thm:virtual} applies to virtual knots of genus one obtained by virtualization at a single crossing. 
\begin{theorem}\label{thm:image}
Let $K$ be a non-trivial classical knot diagram and $v$ be a (classical) crossing on it. If the virtual knot $K_v$ obtained by virtualizing the crossing $v$ is a genus one virtual knot then $K_v$ does not lie in the image of $\overline{v}$.
\end{theorem}
\begin{proof}
%By the theorem of Silver and Williams in \cite{SiWi} we know that $K_v$ is a non-trivial and non-classical knot. 
We analyze the surface state curves of the torus representation of $K_v$ depicted in Figure \ref{fig:virtualized}. 
 We observe that homologically non-trivial state curves are only of the form $2[\mu]$ and $2[\lambda]$. That is, they all have non-trivial multiplicity. It follows from Theorem \ref{thm:virtual} that $K_v$ does not lie in the image of the virtual closure map.
\end{proof}
\begin{figure}[H]
\centering \scalebox{.33}{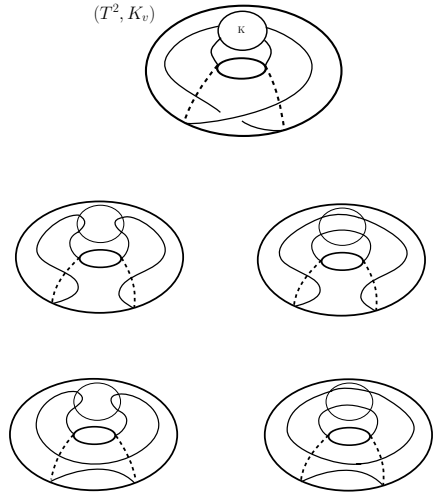}
\caption{The state curves of the virtualization of $K$}
\label{fig:virtualized}
\end{figure}
Corollary \ref{cor:genusone} follows from Theorem \ref{thm:Si} and \ref{thm:image}.

\begin{corollary}\label{cor:genusone}
Let $K$ be a non-trivial classical knot diagram and $v$ be an unknotting crossing. If $K_v$ is the virtual knot obtained by virtualizing the crossing $v$ then $K_v$ is a virtual knot of genus $1$ that does not lie in the image of $\overline{v}$.
\end{corollary}

 The discussion about the virtual closure map brings us to the following question: \textit{Is there a way to classify virtual knots of genus $1$ that are in the image of $\overline{v}$?} 

Korablev and May give an answer to this question in \cite{Ko} by using the geometric degree of knots in thickened surfaces.  
\begin{definition}\normalfont
A knot $K$ in the thickened torus $T^2 \times I$ is said to have \textit{geometric degree one} if there is a vertical annulus $A \subset T^2 \times I$ such that $K$ intersects $A$ at only one point. 
\end{definition}
They show that a virtual knot of genus one is in the image of the the virtual closure map (or according to their terminology the {\it lifting map}) if and only if it is a geometric degree one knot. Moreover, the virtual closure map is injective when it is restricted to set of all prime knotoids of height at least two \cite{Ko}.

Then the question we have asked is equivalent to the question: \textit{How do we detect knots in the thickened torus that have geometric degree one?}. It would be interesting to explore invariants of knots in thickened surfaces detecting the geometric degree. Our approach with the surface bracket states to detect the image of the virtual closure map proposes a direct tool for recognizing virtual knots lying in the image of the virtual closure map. Here we conjecture the following.
\begin{conjecture}\normalfont
A virtual knot of genus one lies in the image of the virtual closure map if and only if all of the surface bracket states of any of its torus representations have multiplicity at most one. 
\end{conjecture}
%%%%%%%%%%%%%%%%%%%%%%%%%%%
\section{Minimal diagrams of knot-type knotoids and the conjecture of Turaev}\label{sec:minimal}

Turaev conjectures \cite{Tu} that minimal crossing diagrams of a knot-type knotoid have zero height. We prove this conjecture by using the following theorems by Nikonov and Manturov \cite{Ma} that are based on a parity projection map from virtual knots to classical knots. In the following statement $K_1 < K_2$ means that $K_1$ is obtained by deleting some chords in the Gauss diagram of $K_2$. This means that some classical crossings of $K_1$ are made virtual to obtain $K_2$.
\begin{theorem}(Nikonov)\label{thm:Ni}
There is a map $pr$ from minimal genus virtual knot diagrams to classical knot diagrams such that for every knot $K$ we have $pr(K) < K$ and if two diagrams $K_1$ and $K_2$ are related by a Reidemeister move %(performed within the given minimal genus diagram)
 then their images $pr(K_1)$, $pr(K_2)$ are related to each other by a Reidemeister move.
%Minimal representations of virtual knots admit the minimum number of crossings. 
\end{theorem}
\begin{theorem}(Manturov)\label{thm:Man}
For every virtual knot diagram $K$ there exists a classical knot diagram $K^*$, such that  $K^* < K$ and $K$ admits the same Gauss diagram with $K^*$ if and only if $K$ is classical. %That is, the genus of the abstract knot diagram surface associated to $K$ is $0$.
\end{theorem}
\begin{corollary}(Manturov)\label{cor:Man}
Let $K$ be a classical knot. Then the minimal number of classical crossings for virtual diagrams in the virtual equivalence class of $K$ is realized in classical diagrams (up to the detour move). For every non-classical diagram (that is, the genus of the associated abstract knot diagram surface is not zero) in the class of $K$, the number of classical crossings is strictly greater than the minimal number of classical crossings.
\end{corollary}
%\begin{theorem}
%Let $K$ be a virtual knot diagram, whose underlying diagram genus is not minimal in the class of the knot $K$. Then there exists a diagram obtained by deleting the ,
%\end{theorem}

\begin{theorem}\label{thm:conjecture}
Minimal diagrams of non-trivial knot-type knotoids have zero height.
\end{theorem}
\begin{proof}
In the discussion, a minimal diagram refers to a diagram (in classical, knotoid and virtual categories) with minimal number of classical crossings.

 Let $k$ be a knot-type knotoid and assume to the contrary that there is a minimal diagram of $k$, $K_1$ that has a non-zero height. Let $n$ be the number of classical crossings of $K_1$. Then the virtual closure of $K_1$, $\overline{v}(K_1)$ is a virtual knot diagram with $n$ classical crossings and a number of virtual crossings that is at least equal to the height of $K_1$.
 
Since $k$ is a knot-type knotoid, the virtual closure of $k$, $\overline{v}(k)$ is a classical knot, and the virtual knot diagram $\overline{v}(K_1)$ lies in the virtual equivalence class of $\overline{v}(k)$ since the virtual closure map $\overline{v}$ is a well-defined map on the set of knotoids in $S^2$. 

The genus of the closed connected orientable surface $F(K_1)$ obtained by attaching $2$-disks to the boundary components of the abstract knotoid diagram surface associated to $K_1$ can be found by the following formula \cite{GK1}
\begin{center}
$g(F(K_1)) = 1 + \frac{((n-1) - \delta)}{2}$,
\end{center}
where $n$ is the crossing number of $K_1$ and $\delta$ is the number of boundary components of the abstract knotoid diagram. The genus of $F(K_1)$ is zero since $K_1$ has only classical crossings. That is, we have, % that is obtained by attaching bands to the arcs of $K_1$, has genus $0$ since $K_1$ is a planar diagram. In fact we have,
\begin{equation}\label{eqn:genus}
%\hspace{7cm}
%\centering
0 =  1 + \frac{((n-1) - \delta)}{2}.
\end{equation} 
%and $g(K_1)$ holds for the genus of the resulting abstract knotoid diagram.
It is clear that the number of boundary components of the associated abstract knotoid diagram is equal to the the number of planar regions determined by the diagram $K_1$, and the boundary components that are adjacent to the endpoints of $K_1$, are distinct since the endpoints of $K_1$ lie in different regions. For this reason, connecting the endpoints of $K_1$ virtually does not change the number of classical crossings but reduces the number of boundary components by $1$. That is, the number of the boundary components of the abstract knot diagram associated to $\overline{v}(K_1)$ is equal to $\delta - 1$. It is known that the genus of the closed connected orientable surface $F(\overline{v}(K_1))$ obtained by attaching $2$-disks to the boundary components of the abstract knot diagram surface associated with $\overline{v}(K_1)$ can be found by the following formula,
$$g(F(\overline{v}(K_1)))~ = ~ 1 + \frac{(n - (\delta - 1))}{2}.$$
  % since the two different boundary components become the same component by the connection.
 From Equation \ref{eqn:genus} we have,
 $$g(F(\overline{v}(K_1)))= g(F(K_1)) + 1 = 1.$$
 %That is, the virtual knot diagram $\overline{v}(K_1)$ lies in a torus.

Then by Theorem \ref{thm:Ni}, any minimal diagram of $\overline{v}(k)$ has strictly less than $n$ classical crossings. Let $\overline{K}$ be any minimal diagram of $\overline{v}(k)$, and $\overline{K}$ has $m$ classical crossings. We have,% The knot $\overline{v}(k)$ is a classical knot then by Corollary \ref{cor:Man} $\overline{K}$ is necessarily a classical knot diagram since it is a minimal diagram for $\overline{v}(K)$. Also, the knot diagrams $\overline{K}$ and $\overline{v}(K_1)$ are virtually equivalent to each other because they both represent the knot $\overline{v}(k)$. The number of classical crossings in the virtual knot diagram $\overline{v}(K_1)$ is equal to $n$ since $K_1$ has $n$ crossings. The virtual knot diagram $\overline{v}(K_1)$ has underlying genus $1$ by the discussion above. Then from Theorem \ref{thm:Ni} we have,
              $$m < n.$$
 
It is clear that the image of $\overline{K}$ under the map $\alpha$ (recall Section \ref{sec:basics}) is a knot-type knotoid diagram with $m$ crossings. Let us denote this knot-type knotoid diagram with $\overline{K}^*$. The underpass closures of the knotoid diagrams $\overline{K}^*$ and $K_1$ are isotopic to the knot $\overline{v}(k)$. This implies that the knotoid diagrams $\overline{K}^*$ and $K_1$ are equivalent to each other since the underpass closure is a bijection map on the knot-type knotoids. We have assumed that $K_1$ is a minimal diagram of $k$. Therefore we have,
$$n \leq m,$$ 
which contradicts with the above inequality. Therefore the assumption is wrong and the theorem follows.
\end{proof}

It is clear that the crossing number of a knot-type knotoid can be at most as the crossing number of the classical knot associated by the map $\alpha$. Turaev asked \cite{Tu} whether the crossing numbers are actually equal. From Theorem \ref{thm:conjecture} we can deduce the following.% implies that a minimal diagram of a knot-type knotoid can be obtained from a minimal diagram of the associated knot via the map $\alpha$ and so it enables us to deduce the following. 
\begin{theorem}
The crossing number of a knot $\kappa$ in $S^3$ is equal to the crossing number of the knot-type knotoid $\kappa^*$ obtained from $\kappa$ via the map $\alpha$. 
\end{theorem}
\begin{proof}
Theorem \ref{thm:conjecture} implies that a minimal diagram of a knot-type knotoid can be obtained from a minimal diagram of the associated knot via the map $\alpha$.

\end{proof}

%\section{Discussion}

%\begin{acknowledgement}\normalfont
\noindent {\bf Acknowledgement.}
Kauffman's work was supported by the Laboratory of Topology and Dynamics, Novosibirsk State University (contract no. 14.Y26.31.0025 with the Ministry of Education and Science of the Russian Federation).
%\end{acknowledgement}


\begin{thebibliography}{99}


%\bibitem{a1} Goundaroulis D., private communication.
\bibitem{Ba} A.~Bartholomew. Andrew Bartholomew's Mathematics Page: Knotoids, \url{http://www.layer8.co.uk/maths/knotoids/index.htm}, January 14, 2015

\bibitem{CKS} J. Scott Carter, S. Kamada, M. Saito. Stable Equivalence of Knots and Virtual Knot Cobordisms,Knots 2000 Korea, Vol. 1 (Yongpyong). {\em J. Knot Theory Ramifications}, {\bf 11}, (2002), no. 3, 311-322

\bibitem{DK1} H.A. Dye, Louis H. Kauffman. Minimal Surface Representations of Virtual Knots and Links, {\em Algebr. Geom. Topol.}, {\bf 5}, (2009):509-535
						
					
\bibitem{GDBS}%[Goundaroulis et al. 2017] 
D.Goundaroulis, J. Dorier, F. Benedetti, A. Stasiak.  Studies of global and local entanglements of individual protein chains using the concept of knotoids. {\em Sci. Reports}, {\bf 2017}, {\em 7}, 6309.

\bibitem{GGLDSK}%[Goundaroulis et al. 2017] 
D. Goundaroulis, N. G\"ug\"umc\"u, S. Lambropoulou, J. Dorier, A. Stasiak, L.H. Kauffman. Topological models for open knotted protein chains using the concepts of knotoids and bonded knotoids. {\em  Polymers},  Special issue on Knotted and Catenated Polymers, Dusan Racko and Andrzej Stasiak Eds. {\bf 2017}, {\em 9(9)}, 444, DOI;10.3390/polym9090444.

\bibitem{Dim} 
D. Goundaroulis, 	J.Dorier, A.Stasiak. A systematic classification of knotoids on the plane and on the sphere.
 ArXiv link: https://arxiv.org/abs/1902.07277, {\bf 2019}.

\bibitem{Go} M.Goussarov, M.Polyak, O.Viro. Finite type invariants of classical and virtual knots. {\em Topology}, 39,pp.1045-1068, (2000).

\bibitem{GK1}%[G\"ug\"umc\"u and Kauffman 2017]
 N. G\"ug\"umc\"u, Kauffman, L.H. New invariants of knotoids. {\em European J. of Combinatorics} {\bf 2017}, {\em 65C}, 186-229.

%\bibitem{GL1}%[G\"ug\"umc\"u and Lambropoulou]
 %N.G\"ug\"umc\"u, S. Lambropoulou.  Knotoids, Braidoids and Applications. {\em Symmetry, Special Issue:Knot Theory and Its Applications} {\bf 2017}, {\em} 
	
	%\bibitem{Gthesis} N. G\"ug\"umc\"u. On Knotoids, Braidoids and Their Applications. {\em PhD thesis}, National Technical University of Athens, {\bf 2017}.
	


	
	
\bibitem{DKK}  H.A.~Dye and A.~Kaestner and L.H.~Kauffman. Khovanov Homology, Lee Homology and a Rasmussen Invariant for Virtual Knots. {\em J. Knot Theory Ramifications}, 26 (2017), no. 3, 1741001, 57 pp. 


%\bibitem{DK}  H.A.~Dye and L.H.~Kauffman, {\it Virtual Crossing Number and the Arrow Polynomial}, J. Knot Theory Ramifications, {\bf 18}, (2009), no.10, 1335-1357

%\bibitem{GP} N.D.Gilbert and T.Porter, {\it Knots and Surfaces}, Oxford University Press, (1994) 

\bibitem{Gr} J.~Green. A table of Virtual Knots, \url{https://www.math.toronto.edu/drorbn/Students/GreenJ/}, August 10, 2004

\bibitem{KK} N.Kamada and S.Kamada. Abstract Link Diagrams and Virtual Knots, J. of Knot Theory and Its Ramifications, {\bf 9},(2000),  no.1, 93-106

\bibitem{Kae} A.~Kaestner. On Applications of Parity in Virtual Knot Theory, PhD thesis, University of Illinois at Chicago, USA, (2011)

\bibitem{kk} A.~Kaestner and L.H. Kauffman. Parity, Skein Polynomials and Categorification, {\em J. of Knot Theory and Its Ramifications}, {\bf 21}, (10/2011), no.13, 56 pp.

\bibitem{Ka1} L.H.~Kauffman. Virtual Knot Theory, {\em European Journal of Combinatorics}, {\bf 20}, (1999), 663-690

\bibitem{Ka2} L.H.~Kauffman, Introduction to Virtual Knot Theory, {\em J. of Knot Theory and Its Ramifications}, {\bf 21}, (2012), no.13, 37 pp.

\bibitem{Ka3} L.H.Kauffman. Detecting virtual knots, {\em Atti. Sem. Mat. Fis. Univ. Modena}, {\bf 49}, (Suppl.), (2001), 241-282


\bibitem{Ka5} L.H.~Kauffman. Knots and Physics, Fourth edition, Series on Knots and Everything, {\bf 53}, World Scientific Publishing Co. Pte. Ltd., Hackensack, NJ, (2013), xviii+846 pp.

%\bibitem{Ka6} L.H.~Kauffman, {\it An affine index polynomial invariant of virtual knots}, J. of Knot Theory and Its Ramifications, {\bf 22}, (2013), no.4, 30 pp.

\bibitem{Ka7} L.H.~Kauffman. New Invariants in the Theory of Knots, {\em Amer. Math. Monthly}, {\bf 95}, (1988), 195-242

\bibitem{Kauf} Louis H.~Kauffman, math.GT/0405049, A self-linking invariant of virtual
knots, {\em Fund. Math.} {\bf 184} (2004), 135--158.


%\bibitem{KM} L.H.~Kauffman and V.~Manturov, {\it Virtual Biquandles}, Fund. Math., {\bf 188}, (2005), 103-146

%\bibitem{Kaw} A.~Kawauchi, {\it A Survey of Knot Theory}, Birkh{\"a}user Verlag,Basel, (1996)

\bibitem{Ku} G.~Kuperberg. What is a virtual Link?, {\em Algebraic\& Geometric Topology}, {\bf 3}, (2003), 587-591

\bibitem{Ko} Ph. G. Korablev, K.Ya. May. Knotoids and knots in the thickened torus, {\em Siberian Mathematical Journal}, Vol:78, No:5, pp.837-844, (2017)

%\bibitem{KnI} C.~Livingston and J.C.~Cha, {\it A table of Knot Invariants}, \url{http://www.indiana.edu/~knotinfo/} 

\bibitem{Ma} V. O. Manturov. Parity in knot theory, (Russian) {\em Mat. Sb.}, {\bf 201}, (2010), no.5, 65-110; translation in {\em Sb. Math.}, {\bf 201}, (2010), no. 5, 693-733

%\bibitem{Man} V.~O.~Manturov, {\it Knot Theory}, Chapman \& Hall/ CRC Press, Boca Raton, FL, (2004)


%\bibitem{Ma4} V.~O.~Manturov, {\it Long Virtual Knots and Their Invariants}, J. Knot Theory Ramifications, {\bf 13}, (2004),1029-1039


\bibitem{Ma3} V.O.~Manturov and D.P.Ilyutko. The State of Art: Virtual Knots, {\em Series on Knots and Everything}:{\bf 51}, World Scientific Publishing Co.Pte. Ltd., Hackensack, NJ, (2013)

%\bibitem{Mi} Y. Miyazawa, {\it A multivariable polynomial invariant for unoriented virtual knots and links}, Journal of Knot Theory and Its Ramifications, {\bf 17}, (2008), no.11, 1311-1326

%\bibitem{Ne} S.~Nelson,{\it Unknotting virtual knots with Gauss diagram forbidden moves}, J. Knot Theory Ramifications, {\bf 10}, (2001), no.6, 931-935. 

%\bibitem{Po} M.~Polyak, {\it Minimal Generating Sets Of Reidemeister Moves}, Quantum Topology, {\bf 1}, (2010), 399-411

%\bibitem{RR} R. C. Read and P.Rosenstiehl, {\it On the Gauss crossing problem} Combinatorics (Proc. Fifth Hungarian Colloq., Keszthely, (1976), Vol.\textbf{II}, 843-876, Colloq. Math. Soc. J´anos Bolyai, {\bf 18}, North-Holland, Amsterdam-New York

%\bibitem{Ro} C.~Rourke, {\it What is a welded link?}, Intelligence of low dimensional topology, (2006), 263-270, 
%Ser. Knots Everything, {\bf 40}, World Sci. Publ., Hackensack, NJ, (2007) 

%\bibitem{Sa} S.~Satoh, {\it Virtual knot presentation of ribbon torus-knots}, J. Knot Theory Ramifications, {\bf 9}, (2000), no. 4, 531-542.

%\bibitem{Tu}  V.~Turaev, {\it Knotoids},  Osaka Journal of Mathematics, {\bf 49}, (2012), no.1 195-223

%\bibitem{Va} V.A.~Vassiliev, {\it Cohomology of knot spaces}, Theory of singularities and its applications,  Adv. Soviet Math.,{\bf 1}, Amer. Math. Soc., Providence, RI, (1990), 23-69
\bibitem{SiWi} D. Silver, S. Williams. On a class of virtual knots with unit Jones polynomial. {\em Journal of Knot Theory and its Ramifications}, 13(3), May 2004\\
DOI: 10.1142/S0218216504003196

\bibitem{Tu}
V.Turaev. Knotoids. {\em Osaka Journal of Mathematics}  {\bf 2012}, {\em 49}, 195-223.



\end{thebibliography}
\end{document}